\documentclass[11pt,amstex,amssymb]{amsart}
\usepackage{amsmath,amsthm,amsfonts,amssymb,amscd}
\usepackage[latin1]{inputenc}
\usepackage[all]{xy}
\usepackage[dvips]{graphicx}
\usepackage{amsmath}
\usepackage{amsthm}
\usepackage{amsfonts}
\usepackage{amssymb}

\newtheorem{Theorem}{Theorem}[section]
\newtheorem{Lemma}[Theorem]{Lemma}

\newtheorem{Proposition}[Theorem]{Proposition}

\newtheorem{Claim}[Theorem]{Claim}

\def\fa{{\mathcal{F}}}

\def\ord{\operatorname{{ord}}}

\def\ov{\overline}

\def\bc{{\mathbb{C}}}
\def\bn{{\mathbb{N}}}

\def\GL{\operatorname{{GL}}}

\def\Diff{\operatorname{{Diff}}}

\def\Hol{\operatorname{{Hol}}}

\def\Id{\operatorname{{Id}}}

\def\Diff{\operatorname{{Diff}}}

\def\SO(3){\operatorname{{SO(3)}}}

\def\bc{\operatorname{{\mathbb C}}}

\def\Hol{\operatorname{{Hol}}}
\def\fa{\operatorname{{\mathcal F}}}
\def\Diff{\operatorname{{Diff}}}

\def\ov\bc{\operatorname{\overline{\mathbb{C}}}}

\title[Stable compact leaves]{Meagerness of the set of   compact leaves for transversely holomorphic foliations}

\author{Bruno Scardua}
\begin{document}

\maketitle

\section{Introduction}

Some of the  most important important aspects of the geometry of a foliation are linked to its set of compact leaves. This is subject of remarkable achievements  like
the celebrated  stability theorems of Reeb (\cite{Godbillon, Reeb1,Camacho-LinsNeto}). Motivated by these we ask about the largeness of the set of compact leaves.   One important question in this direction is the following: {\it If a codimension one smooth foliation on a compact
manifold has  infinitely many compact leaves, then is it true that
all leaves are compact?} The answer is clearly no, but this  is true
for (transversely) real analytic foliations of codimension one on compact manifolds (\cite{haefliger}).
On the other hand, there are versions of  Reeb stability results  for the class of {\em holomorphic}
 foliations (\cite{Brunella}). In the holomorphic framework,  it is proved in
\cite{Brunella-Nicolau}
 that a (non-singular) transversely holomorphic codimension one
 on a compact connected manifold admitting
infinitely many compact leaves exhibits a transversely meromorphic
first integral. The problem of bounding the number of closed leaves of a holomorphic
foliation is known (at least in the complex algebraic framework) as
{\it Jouanolou's problem}, thanks to the pioneering results in
\cite{Jouanolou} and has a wide range of contributions and
applications in the Algebraic setting.

 Motivated by the above discussion,   in \cite{Camacho-Scarduameasure},  we focus on the problem of  existence
of a suitable compact leaf under the hypothesis of existence of a
positive measure  set of  of compact leaves.
In this paper we pursue the study of the structure of the set of compact leaves, by proving its meagerness provided that there are no stable compact leaves. We state it as follows:

\begin{Theorem}
\label{Theorem:measuremainfoliations} Let $\fa$ be a transversely
holomorphic foliation on a  connected complex manifold $M$.
 Denote by $\Omega(\fa)\subset M$ the subset of compact
leaves of $\fa$.  Then we have two possibilities:
\begin{itemize}

\item[{\rm (i)}] $\fa$ has some compact leaf with
finite holonomy group and therefore $\Omega(\fa)$ contains some open subset.
\item[{\rm(ii)}] The set $\Omega(\fa)$ is meager. Indeed, $\Omega(\fa)$ is contained in a countable union of
proper analytic subsets of $M$. 
\end{itemize}
\end{Theorem}
Recall that a subset $X \subset M$ is called {\it meager} if
it is contained in a countable union of closed subsets with empty interior in $M$. Equivalently, its complementary $M\setminus X$ is {\it residual}, meaning that it contains a countable intersection of dense open subsets of $M$.
A compact leaf with finite holonomy group will be called {\it stable} (cf. \cite{Godbillon}). In view of the Reeb local stability theorem (\cite{Camacho-LinsNeto, Godbillon, Reeb1}, a stable leaf always belongs to the interior of the set of $\Omega(\fa)$, union of compact leaves, therefore  Theorem~\ref{Theorem:measuremainfoliations} can be stated as:

{\it A transversely holomorphic foliation on a compact complex manifold, exhibits a compact stable
leaf if and only if the set of compact leaves is not a meager  subset of  the manifold}.

The corresponding version for groups of diffeomorphisms of a complex manifold is:

\begin{Theorem}
\label{Theorem:maingroups}
 Let $G\subset \Diff(F)$ be a  subgroup
of holomorphic diffeomorphisms of a complex connected manifold $F$.
Denote by $\Omega(G)$ the subset of points $x \in F$ such that the
$G$-orbit of $x$ is finite. There are two mutually exclusive possibilities:

\begin{itemize}

\item[{\rm (i)}] $G$ is a finite group and $\Omega(G)=F$.
\item[{\rm(ii)}] The set $\Omega(G)$ is meager.
\end{itemize}

\end{Theorem}

\section{Periodic holonomy groups and holonomy groups of finite exponent}
\label{section:generalities}

\subsection{Holonomy groups}

Let $\fa$ be a codimension $k$  holomorphic foliation on a complex
manifold $M$.  Given a point $p\in M$, the leaf through $p$ is
denoted by $L_p$. We denote by $\Hol(\fa,L_p)=\Hol(L_p)$ the
holonomy group of  $L_p$. This is a conjugacy class of equivalence,
and we shall denote by $\Hol(L_p,\Sigma_p,p)$ its representative
given by the local representation of this holonomy calculated with
respect to a local transverse section $\Sigma_p$ centered  at the
point $p\in L_p$. The group  $\Hol(L_p, \Sigma_p,p)$ is therefore a
subgroup of the group of germs $\Diff(\Sigma_p,p)$ which is
identified with the group $\Diff(\bc^k,0)$ of germs at the origin
$0\in \mathbb C^k$ of complex diffeomorphisms.


\subsection{Groups of finite exponent}

Next we present Burnside's and Schur's results on periodic linear
groups.  Let $G$ be a group with identity $e_G\in G$. The group is
{\it periodic} if
 each element of $G$ has finite order.
 A periodic group $G$ is {\it periodic of bounded exponent} if
there is an uniform upper bound for the orders of its elements. This
is equivalent to the existence of $m \in \mathbb N$ with $g^m = 1$
for all $g \in G$ (cf. \cite{Santos-Scardua}).  Because of this, a
group which is periodic of bounded exponent is also called a group
of {\it finite exponent}.
 The following classical results about finite exponent linear groups are  due to Burnside and Schur.

\begin{Theorem}[Burnside, 1905 \cite{Burnside}, Schur, 1911
\cite{Schur}]
\label{Theorem:Burnside} Let $G\subset \GL(k,\bc)$ be a complex
linear group.
\begin{enumerate}

\item[{\rm(i)}]{\rm(Burnside)} If $G$ is of finite exponent $\ell$  {\rm(}but not
necessarily finitely generated{\rm)} then $G$ is finite; actually we
have $|G| \le \ell^{k^2}$.

\item[{\rm(ii)}] {\rm(Schur)} If $G$ is  finitely generated and periodic
{\rm(}not necessarily of bounded exponent{\rm)} then $G$ is finite.

\end{enumerate}
\end{Theorem}
\noindent Now we shall use the above result in the  study of finite exponent groups of germs of complex diffeomorphisms. The First step is the following:

\subsection{Holonomy groups of finite exponent}
Let us collect some results for groups of germs finite exponent. These results proved in \cite{Camacho-Scarduameasure} are stated below:
\begin{Lemma}[\cite{Santos-Scardua,Camacho-Scarduameasure}]
\label{Lemma:finiteexponentgerms} About periodic groups of germs of
complex diffeomorphisms we have:
\begin{enumerate}

\item A finitely
generated periodic subgroup  $G \subset \Diff(\bc^k,0)$ is
necessarily finite.   A {\rm(}not necessarily finitely
generated{\rm)} subgroup
 $G \subset \Diff(\bc^k,0)$ of finite exponent is necessarily finite.

\item  Let $G\subset \Diff(\mathbb C^k,0)$ be a finitely
generated subgroup.  Assume that there is an invariant connected
neighborhood $W$ of the origin in $\mathbb C^k$ such that each point
$x$ is periodic for each element $g \in G$. Then $G$ is a finite
group.

\item Let $G \subset \Diff(\bc^k,0)$ be a {\rm(}not necessarily finitely generated{\rm)} subgroup
such that for each point $x$ close enough to the origin,  the
pseudo-orbit of $x$ is finite   of {\rm(}uniformly bounded{\rm)}
order $\le \ell$ for some $\ell \in \bn$, then $G$ is finite.

\end{enumerate}

\end{Lemma}

\begin{Proposition}[Finiteness lemma, \cite{Camacho-Scarduameasure}]
\label{Proposition:finitenesslemma}
Let $G$ be a  subgroup of holomorphic
diffeomorphisms of a connected complex manifold $F$. Assume that:

\begin{enumerate}

\item $G$ is periodic of finite exponent or $G$ is finitely generated and periodic.

\item $G$ has a finite orbit in $F$.

\end{enumerate}
Then $G$ is finite.
\end{Proposition}

\section{Meagerness versus finiteness}

Let us now prove Theorems~\ref{Theorem:measuremainfoliations} and
\ref{Theorem:maingroups}. For sake of simplicity we will adopt the following notation:
if a subset $X\subset M$ is not a zero measure subset then we shall write $\mu(X)>0$. This may cause no confusion for we are not considering any specific measure $\mu$ on $M$ and we shall be dealing only with the notion of zero measure subset. Nevertheless, we notice that if
$X\subset M$ writes as a countable union $X=\bigcup \limits_{n \in \mathbb N} X_n$ of subsets $X_n\subset M$ then $X$ has zero measure in $M$ if and only if $X_n$ has zero measure in $M$ for {\it all} $n \in \mathbb N$. In terms of our notation we have therefore $\mu(X)>0$ if and only if $\mu(X_n)>0$ for {\it some} $n \in \mathbb N$.

\begin{proof}[Proof of Theorem~\ref{Theorem:measuremainfoliations}]
Let us assume that $\Omega(\fa)$ contains no leaf with finite holonomy and prove that $\Omega(\fa)$ is meager.  Because $M$ is compact there is a finite number of relatively
compact open discs $T_j \subset M, \, j=1,...,r$ such:
\begin{enumerate}
\item Each $T_j$ is transverse to $\fa$ and the closure ${\overline{T_j}}$
is contained in the
interior of a transverse disc $\Sigma_j$ to $\fa$.
\item Each leaf
of $\fa$ intersects at least one of the discs $T_j$.
\end{enumerate}

Put $T=\bigcup\limits_{j=1}^r T_j \subset M$ and define
\[
\Omega(\fa,T)=\{L\in \fa: \# (L \cap T )< \infty \}.
 \]
 Then $\Omega(\fa, T)= \bigcup \limits_{n=1} ^\infty \Omega(\fa,
 T,n)$ where
 \[
\Omega(\fa, T, n)=\{L\in \fa: \# (L \cap T )\leq n \}.
\]

\begin{Claim}
We have $\Omega(\fa)=\Omega(\fa, T)$.
\end{Claim}
\begin{proof}
Indeed, given a leaf $L\in \fa$ if $L\notin \Omega(\fa,T)$ then
there is some $j$ such that $\#(L\cap T_j)=\infty$. Since
$\overline{T_j}$ is compact there is a point $p\in \Sigma_j$
belonging to the closure of $L$ and which is accumulated by points
in $L$. Since $p \in \Sigma_j$ which is transverse to $\fa$ we
conclude that $L$ has infinitely many plaques intersecting any
distinguished neighborhood of $p$ in $M$ and therefore $L$ cannot be
compact. Conversely, suppose that $L\in\Omega(\fa, T)$ then $L$ has
only finitely many plaques in a (finite) covering of $M$ by
distinguished neighborhoods. Since $M$ is compact this implies that
$L$ is compact.
\end{proof}

From now on we shall write $\Omega(\fa,n)$ for $\Omega(\fa,T,n)$ and simply $\Omega(\fa)$ for $\Omega(\fa,T)$.
For the leaves $L\in \Omega(\fa,n)$ we put $\ord(\fa):=n$.

Because $\Omega(\fa)=\bigcup\limits_{n\in \mathbb
N}\Omega(\fa,n)$, it is then enough to show that  and $\Omega(\fa,n)$ is meager for each $n \in \mathbb N$.

Next we claim:
\begin{Claim} For each $n \in \mathbb N$, given a leaf $L\subset \Omega(\fa,n)$, there is an open neighborhood $L\subset W\subset M$, such that $\Omega(\fa,n)\cap W$   is meager.
\end{Claim}
\begin{proof}
Let us fix $n_0 \in \mathbb N$. Given a leaf $L_0 \in \Omega(\fa,n_0)$, there are two possibilities: either  $L_0$ is isolated in $\Omega(\fa,n_0)$ or it is not. If $L_0$ is isolated in $\Omega(\fa,n_0)$ then we  can find an open neighborhood $W(L_0)$ of $L_0$ in $M$ such that $\Omega(\fa,n_0)\cap W(L_0)=L_0$. Assume now that $L_0$ is not isolated in $\Omega(\fa,n_0)$. This means that for each neighborhood $W$ of $L_0$ in $M$, we have infinitely many leaves in  $\Omega(L_0,n_0)$ different from $L_0$.

Since the holonomy of $L_0$ is not finite (we are assuming that $\Omega(\fa)$ has no leaf with finite holonomy group), by Lemma~\ref{Lemma:finiteexponentgerms} there is some holonomy map germ $h \in \Hol(\fa,L_0)$ which is not a finite order map (notice that, since $L_0$ is compact, its fundamental group and, therefore, its holonomy group is finitely generated).
\end{proof}
Let therefore $L_0\in \Omega(\fa,T,n_0)$ be as above. We may choose a
base point $p\in L_0\cap T$ and a transverse disc $\Sigma_p\subset
{\overline\Sigma_p}\subset T$ to $\fa$ centered at $p$. Given a
point $z \in \Sigma_p$ we denote the leaf through $z$ by $L_z$. If
$L_z \in \Omega(\fa,n_0)$ then $ \# (L_z \cap \Sigma_p) \leq n_0$.

Take now the correspondent  of the holonomy map germ $h \in \Hol(\fa, L_0, \Sigma_p,p)$. Let
us choose a sufficiently small subdisc $W\subset \Sigma_p$ such that
the  germ $h$ has a representative $h\colon W \to \Sigma_p$ such
that the iterates  $h,h^2,...,h^{n_0  !}$ are defined in $W$. Since any leaf $L\in \Omega(\fa,n_0)$ intersects $W$ in at most $n_0$ points, we have that either this intersection is empty or given any point $z \in L\cap W$ we have $h^{n_0!}(z)=z$.
of the claim above we have $h^{\ell(z)}(z)=z$ for some $1\leq \ell(z)\leq n_0$.
For each $\ell \in \{1,...,n_0\}\subset \mathbb N$ put $X(\ell,h)=:\{z\in W: h^\ell (z) = z\}$.
Since $h$ is analytic, $X(\ell,h)$ is an analytic subset of $W$. Because $h$ is not a finite order map, $X(\ell,h)$ is a proper analytic subset of $W$, in particular it as closed subset of  real codimension $\geq 2$ with  empty interior.  From the above considerations we conclude that
\[
\Omega(\fa,n_0)\cap W \subset \bigcup\limits_{\ell=1} ^{n_0} X(\ell,h)
\]

\begin{Claim}
\label{Lemma:boundedorderclosed}
Let a leaf $L\in\mathcal{F}$ be such that
$L\subset\overline{\Omega(\mathcal{F},n)}$ for some $n\in
\mathbb{N}$. Then $L$ is compact, indeed $L\subset\Omega(\fa,n)$. In other words, the union of leaves in $\Omega(\fa,n)$ is a closed subset of $M$.
\end{Claim}

\begin{proof}
Let $L\subset\overline{\Omega(\mathcal{F},n)}$. Since $M$ is compact, if $L$ is not compact, it is not closed.
Suppose then by contradiction that $L$ is not closed.  There is an accumulation point $q_{\infty}\in\overline{L}\setminus
L$. Given an
{\it arbitrarily small} transverse disc $\Sigma_{q_{\infty}}$ to
$\mathcal{F}$ centered at $q_{\infty}$, we have $\#(L\cap\Sigma_{q_{\infty}%
})=\infty$. Then there is a disc $T_{j}\subset T$ such that $L_{q_{\infty}}$
meets $T_{j}$ at an interior point say, $q\in T_{j}$. Choose now a point $p\in
L$ and a transverse disc $\Sigma_{p}$ centered at $p$. By the Transverse
uniformity lemma (\cite{Camacho-LinsNeto}), there is a map from the
disc $\Sigma_{p}$ to $\Sigma_{q_{\infty}}$ and thus to the disc $T_{j}$. We
conclude that, given $k \in \mathbb N$,  for any $w\in\Sigma_{p}$ close enough to $p$, we have
$\#(L_{w}\cap T_{j})\geq k+1$. In particular, $L_{w}$ has order greater than
$k$. Since $\Omega(\fa,n)\cap \Sigma_p$ has an accumulation point at the origin $p\in \Sigma_p$, we get a contradiction. This shows that $L$ is closed, therefore compact. Indeed, the above argumentation shows that if $\ord(L) \leq n$, so that indeed we have $L\in \Omega(\fa,n)$.
\end{proof}

Using the two claims above we can conclude that
for each $n \in \mathbb N$, there is an open cover of $M$ by open sets $W$ such that $\Omega(\fa,n)\cap W$   is meager.
Since $M$ is finite, this open cover can be chosen to be finite and then this shows that $\Omega(\fa,n)$ is a meager subset of $M$.
\end{proof}

\begin{proof}[Proof of Theorem~\ref{Theorem:maingroups}]
 If $\Omega(G)=\emptyset$ then it is meager. So we may assume that $G$ has some periodic orbit. Then, thanks to Burnside's theorem (\ref{Theorem:Burnside}) and and Proposition~\ref{Proposition:finitenesslemma} (where the existence of a periodic orbit is required) it is enough to prove the following
 claim:
\begin{Claim} If $\Omega(G)$ is not meager then  $G$ is a periodic group of finite exponent.
\end{Claim}
\begin{proof}[proof of the claim]
Assume that $\Omega(G)$ is not meager.
We have $\Omega(G)=\{x \in F: \# \mathcal O_G(x)< \infty\}=
\bigcup\limits_{k=1} ^\infty \{x \in F: \# \mathcal O_G(x)\leq k\}$,
therefore there is some $k\in \mathbb N$ such that
$\Omega(G,k):=\{x \in F: \# \mathcal O_G(x) \leq k\}$ is not meager.
Given any   diffeomorphism $f \in G$,  the set $\Omega(f,k):=\{x \in F: \# \mathcal O_f(x) \leq k\}\supset \Omega(G,k)$ is not meager. Since $\Omega(f,k)\subset \bigcup \limits_{\ell=1}^k \{x \in F: f^{\ell}(x)=x\}$,  there is $k_f \leq k$ such that the set $X=\{x \in F:
f^{k_f}(x)=x\}$ is not meager. Since $X\subset F$ is an
analytic subset, this implies that $X=F$ (a proper analytic subset
of a connected complex  manifold  (has real codimension $\geq 2$ and therefore it) is a meager subset of  $F$). Therefore, we have $f^{k_f}=\Id$ in $F$. This shows that $G$ is periodic of finite
exponent.
\end{proof}
\end{proof}


\begin{thebibliography}{99}



\bibitem{Brunella} Brunella, Marco. {\it A global stability
theorem for transversely holomorphic foliations}.
Ann. Global Anal. Geom. 15 (1997), no. 2, 179--186.

\bibitem{Brunella-Nicolau} Brunella, Marco; Nicolau, Marcel.
{\it Sur les hypersurfaces solutions
des \'equations de Pfaff}.  C. R. Acad. Sci. Paris S\'er. I Math. 329 (1999), no. 9,
793--795.


\bibitem{Burnside}  Burnside, W.: {\it On criteria for the finiteness of
the order of a group of linear substitutions}, Proc.London Math.
 Soc. (2) 3 (1905), 435-440.



\bibitem{Camacho-LinsNeto} Camacho, C\'esar; Lins Neto, Alcides. {\it Geometry theory of
foliations.} Translated from the Portuguese by Sue E. Goodman.
Birkh\"auser Boston, Inc., Boston, MA, 1985. vi + 205 pp.

\bibitem{Camacho-Scarduameasure}{\ C\'esar Camacho, Bruno Sc\'ardua};
{\it On the existence of stable compact leaves for transversely holomorphic foliations}. Topology and its Applications 160(13):1802 (2013).




\bibitem{Godbillon} Godbillon, Claude. {\it Feuilletages} (French) [Foliations] \'Etudes
g\'eom\'etriques [Geometric studies]. With a preface by G. Reeb.
Progress in Mathematics, 98. Birkh\"auser Verlag, Basel, 1991. xiv +
474 pp.

\bibitem{haefliger} A. Haefliger, NAISSANCE DES FEUILLETAGES,
D'EHRESMANN-REEB \`A NOVIKOV. www.unige.ch/math/folks/haefliger/Feuilletages.pdf


\bibitem{Jouanolou}   Jouanolou, Jean-Pierre. {\it \'Equations de Pfaff alg\`ebriques}; Lecture Notes in Math. 708, Springer-Verlag, Berlin, 1979.



\bibitem{Reeb1}   Reeb, Georges. {\it
Vari\'et\'es feuillet\'ees, feuilles voisines}; C.R.A.S. Paris 224
(1947), 1613-1614.

\bibitem{Santos-Scardua} Santos, F\'abio; Scardua, Bruno. {\it Stability of complex  foliations transverse to fibrations}, Proc. Amer. Math. Soc. 140 (2012), no. 9, 3083-3090.


\bibitem{Schur}  Schur, I.: {\it \"Uber Gruppen periodischer substitutionen},
Sitzungsber. Preuss. Akad. Wiss. (1911), 619--627.


\end{thebibliography}
\end{document}